\documentclass[12pt,a4paper,reqno]{amsart}
\usepackage{amsfonts}
\usepackage{amsthm}
\usepackage{amsmath}
\usepackage{amssymb}
\usepackage{mathabx}
\usepackage{bm}
\usepackage{amscd}
\usepackage{t1enc}
\usepackage[mathscr]{eucal}
\usepackage{indentfirst}
\usepackage{graphicx}
\usepackage{graphics}
\usepackage{pict2e}
\usepackage{epic}
\numberwithin{equation}{section}
\usepackage[margin=2.9cm]{geometry}
\usepackage{epstopdf}
\usepackage[colorlinks,linkcolor=blue]{hyperref}
\usepackage{color}
\usepackage[capitalise]{cleveref}

\usepackage{enumitem}
\usepackage{cases}

\crefformat{equation}{(#2#1#3)}
\crefrangeformat{equation}{(#3#1#4) to~(#5#2#6)}

\usepackage[normalem]{ulem}

\allowdisplaybreaks

\theoremstyle{plain}
\newtheorem{theorem}{Theorem}[section]
\newtheorem*{theorem*}{theorem}
\newtheorem{lemma}[theorem]{Lemma}
\newtheorem{corollary}[theorem]{Corollary}
\newtheorem{proposition}[theorem]{Proposition}

\theoremstyle{definition}

\newtheorem{remark}[theorem]{Remark}

\newcommand{\dv}{\mathrm{div}}

\newcommand{\vu}{\mathbf{u}}
\newcommand{\vn}{\mathbf{n}}

\newcommand{\sL}{\mathscr{L}}

\newcommand{\subeqref}[2]{$ \eqref{#1}_{#2} $}

\newcommand{\abs}[2]{\left\lvert #1 \right\rvert^{#2}}

\begin{document}

	\title[]{Immediate blowup of entropy-bounded classical solutions to the vacuum free boundary problem of non-isentropic compressible Navier--Stokes equations}
	
	\author[X. Liu]{Xin LIU}
	\address[X. Liu]{
		Department of Mathematics, Texas A\&M University \\
		College Station TX 77843-3368, Texas, USA. 
	}
	\email{stleonliu@gmail.com}
	
	\author[Y. Yuan]{Yuan YUAN$ ^* $}
	\address[Y. Yuan]{School of Mathematical Sciences, South China Normal University \\
		Guangzhou 510631, China.}
	\email{yyuan2102@m.scnu.edu.cn}
	\thanks{$ ^* $Corresponding author.}
	
	\date{\today}
	
	\begin{abstract}
		This paper considers the immediate blowup of entropy-bounded classical solutions to the vacuum free boundary problem of non-isentropic compressible Navier--Stokes equations. The viscosities and the heat conductivity could be constants, or more physically, the \emph{degenerate, temperature-dependent} functions which vanish on the vacuum boundary (i.e., $\mu=\bar{\mu} \theta^{\alpha}, ~ \lambda=\bar{\lambda} \theta^{\alpha},~ \kappa=\bar{\kappa} \theta^{\alpha}$, for constants $0\leq \alpha\leq 1/(\gamma-1)$, $\bar{\mu}>0,~2\bar{\mu}+n\bar{\lambda}\geq0,~\bar{\kappa}\geq 0$, and adiabatic exponent $\gamma>1$). 
		
		With prescribed decaying rate of the initial density across the vacuum boundary, we prove that: (1) for three-dimensional spherically symmetric flows with \emph{non-vanishing bulk viscosity and zero heat conductivity}, entropy-bounded classical solutions \textbf{do not exist} for any small time, provided the initial velocity is expanding near the boundary; 
		(2) for three-dimensional spherically symmetric flows with \emph{non-vanishing heat conductivity}, the normal derivative of the temperature of the classical solution across the free boundary \textbf{does not degenerate}, and therefore the entropy \textbf{immediately blowups} if the decaying rate of the initial density is not of $1/(\gamma-1)$ power of the distance function to the boundary; 
		(3) for one-dimensional flow with \emph{zero heat conductivity}, the non-existence result is similar but need more restrictions on the decaying rate. 
		
		Together with our previous results on local or global entropy-bounded classical solutions (\textit{Liu and Yuan, SIAM J. Math. Anal. (2) 51, 2019}; \textit{Liu and Yuan, Math. Models Methods Appl. Sci. (9) 12, 2019}), this paper shows the necessity of proper degenerate conditions on the density and temperature across the boundary	for the well-posedness of the entropy-bounded classical solutions to the vacuum boundary problem of the viscous gas.
		
	\end{abstract}
	
\keywords{compressible Navier--Stokes equations, vacuum free boundary, blowup, maximum principle}

\subjclass[2010]{35Q30, 74A15, 76N10}

	\maketitle


\section{Introduction and main results}

\subsection{Equations}
In this paper, we study the motion of non-isentropic viscous gas connecting vacuum via free boundary. The gas occupied domain is denoted by $\Omega(t) \subset \mathbb{R}^n $, $ n = 1,2,3$, which is assumed to be a simply connected, open domain evolving with the gas flow. 

In $\Omega(t)$, the non-isentropic viscous gas flow is modeled by the non-isentropic compressible Navier--Stokes equations (CNS):
\begin{equation}\label{NS}
	\left\{
	\begin{aligned}
		& \rho_t+\dv\, (\rho \vu)=0 \, , \\
		& (\rho \vu)_t +\dv \, (\rho \vu \otimes \vu)+\nabla p=  \dv \, (\mu (\nabla \vu+(\nabla \vu)^\intercal)+\lambda  \dv  \vu \ \mathbb{I}_n) \, ,\\
		&c_v\partial_t (\rho \theta)+c_v\dv\, (\rho \vu \theta) + p \dv \vu= \frac{\mu}{2}  \abs{\nabla \vu + (\nabla \vu) ^\intercal}{2}  + \lambda (\dv \vu)^2 + \kappa \Delta \theta \, ,   
	\end{aligned}
	\right.
\end{equation}
where the space-time variables are $(y,t)\in \mathbb{R}^n\times[0,\infty)$, and $\rho$, $\vu$, $p$, and $\theta$ represent the scalar density, the velocity field, the pressure potential, and 
the absolute temperature, respectively. 
$\mu$, $2\mu+n\lambda$ and $\kappa$ are the shear viscosity, the bulk viscosity and the heat conductivity, respectively and in this paper they are assumed to be constants, or  \emph{temperature-dependent functions which vanish on the vacuum boundary}:
\begin{equation}\label{coefficients}
	\mu=\bar{\mu} \theta^{\alpha}, ~ \lambda=\bar{\lambda} \theta^{\alpha},~ \kappa=\bar{\kappa} \theta^{\alpha},
	~ \text{where~constants~} \bar{\mu}>0,~2\bar{\mu}+n\bar{\lambda}\geq0,~\bar{\kappa}\geq 0, ~\alpha\geq0.
\end{equation}
\emph{Here $\mu,\lambda,\kappa$ are constants when $\alpha=0$.} Such a setting of the viscosities agrees with the kinetic theory of gas dynamic, in which the CNS is derived from the Boltzmann equations through the Chapman--Enskog expansion.
More precisely, if the intermolecular potential varies as $r^{-a}$ with $r$ being the molecule distance and $a$ being a positive constant, then $\mu, \lambda$ satisfy
\begin{equation}\label{coefficients-1}
	\mu=\bar{\mu} \theta^{\frac{a+4}{2a}}, ~ \lambda=\bar{\lambda} \theta^{\frac{a+4}{2a}}, ~ \kappa=\bar{\kappa} \theta^{\frac{a+4}{2a}}.
\end{equation}
Obviously, our setting \eqref{coefficients} is the general form of \eqref{coefficients-1}.
We refer the interested readers to Chapman--Cowling \cite{Chapm.C1970} and Li--Qin \cite{Li.Q2014} for the Chapman--Enskog expansion.

We assume that the gas is polytropic. The pressure potential $p$, the specific inner energy $ e $, and the entropy $ s $ are related by the following equations of state, satisfying the Gibbs--Helmholtz equation:
\begin{equation}
	\label{state}
	p=R\rho\theta = \bar{A}e^{s(\gamma-1)/R}\rho^{\gamma},\quad
	e=c_v\theta,
\end{equation}
where $\gamma>1$, $R>0$, and $ c_v=\frac{R}{\gamma-1}>0$ are referred to as the adiabatic exponent, the universal gas constant, and the specific heat coefficient, respectively. Here $ \bar A $ is a positive constant.

On the moving gas--vacuum interface $ \Gamma(t):=\partial \Omega(t) $, the normal stress balance condition is given by
\begin{equation}\label{boundary}
	(\mu (\nabla \vu+(\nabla \vu)^\intercal)+\lambda  \dv \vu \ \mathbb{I}_n -p\mathbb{I}_n) \vn=0, ~ \quad\text{on} \ \Gamma(t),
\end{equation}
and $\Gamma(t) $ evolves in time with the gas flow on the boundary, i.e.,
\begin{equation}\label{kinematic}
	\mathcal{V}(\Gamma (t))= (\vu \cdot \vn) \vert_{\Gamma(t)},
\end{equation}
where $\mathcal{V}(\Gamma(t))$ and $ \vn $ are the normal velocity of the evolving interface and the exterior unit normal vector on $ \Gamma(t) $, respectively. 
System \eqref{NS} with \eqref{boundary} is complemented with the following initial data,
\begin{equation}
	\label{initial}
	(\rho, \vu, \theta)(y,t=0)= (\rho_0,\vu_0,\theta_0)(y), \quad y \in \Omega_0:=\Omega(0),
\end{equation}
where $\rho_0,\theta_0>0$ in $\Omega_0$ and $\rho_0=\theta_0=0$ on $\Omega_0$.

In this paper, we investigate solutions to system \eqref{NS} with the following properties:
\begin{enumerate}[label = p-\theenumi.,ref = p-\theenumi]
	\item Initially density connects to vacuum continuously; \label{property:vacuum}
	\item Entropy remains bounded as long as solutions exist. \label{property:entropy}
\end{enumerate}
That is, we consider solutions to system \eqref{NS} with bounded entropy and vacuum free boundary. In particular, property \ref{property:vacuum} implies that, thanks to the continuity equation \subeqref{NS}{1}, density $ \rho $ vanishes on the moving boundary $\Gamma(t)$. Additionally, property \ref{property:entropy} implies that, thanks to \eqref{state}, temperature $ \theta $ vanishes on $ \Gamma(t) $ as well. 
Therefore, we have 
\begin{equation}
	\label{boundary-2}
	 \rho=0,~ \theta=0 ~ \quad\text{on} \ \Gamma(t),
\end{equation} 
provided our solutions with properties \ref{property:vacuum} and \ref{property:entropy} exist. 

Moreover, we consider the initial density profile with
\begin{equation}\label{H1}
	 - \infty< \nabla_\vn (\rho_0^{1/\delta}) < 0 \quad \text{on~} \Gamma_0 := \partial\Omega_0  
\end{equation}
for a positive constant $\delta$,  where $\nabla_\vn$ is the outward normal derivative. 
When $\delta=1/(1-\gamma)$, such a \emph{singular boundary condition} for the initial density profile is exactly the \emph{physical vacuum condition} (see \eqref{physical-vacuum}) in the isentropic case.
Thanks to property \ref{property:vacuum}, by denoting the distance function to the boundary $ \Gamma_0 $ as $d(y)$,  \eqref{H1} implies that, for $ y \in \Omega_0 $ in a neighborhood of $ \Gamma_0 $, 
\begin{equation}\label{H1-1} 
\left\{
\begin{gathered}
	C_1^{-1} d^{\delta}(y) <  \rho_0(y) < C_1 d^{\delta}(y),\\
	C_2^{-1} d^{\delta-1}(y) <  \abs{\nabla_\vn \rho_0(y)}{} < C_2 d^{\delta-1}(y) 
\end{gathered}
\right.
\end{equation}
for some positive constants $ C_1,C_2 \in (0,\infty)$.

Without loss of generality, we consider solutions to \eqref{NS} with properties \ref{property:vacuum} and \ref{property:entropy} in the following two cases:
\begin{enumerate}[label = Case \theenumi., ref = Case \theenumi]
	\item $ n = 3 $ with spherical symmetry, and $\Omega_0 = B((0,0,0),1)$, i.e., unit ball centered at the origin;
	\item $ n = 1 $ and $ \Omega_0 = (0,1) $.
\end{enumerate}

\subsection{Literature review}
There have been a huge number of literatures concerning the CNS. For the Cauchy and first initial boundary value problems, when there are positive lower bounds of the initial density profiles (i.e., $\rho\geq \underline\rho > 0 $), the local well-posedness of classical solutions has been investigated by Serrin \cite{Serri.1959}, Nash\cite{Nash.1962}, Itaya \cite{Itaya.1971}, Tani \cite{Tani.1977}. The pioneering works of Matsumura and Nishida \cite{Matsu.N1980,Matsu.N1983} showed the global stability of equilibria for the heat conductive flows with respect to small perturbations. 
Later, Hoff and Smoller \cite{Hoff.S2001} proved that vacuum states do not occur in the gas described by the one-dimensional CNS, provided no vacuum states are present initially. Such a conclusion is also true for three-dimensional and spherically symmetric case away from the origin (\cite{Guo.L.X2012a}), however, it is not confirmed without the symmetry in multi-dimensional case.

When the density profile contains vacuum state (i.e., $\rho\geq0$), Cho and Kim \cite{Cho.K2006,Cho.K2006a} showed the local well-posedness of the CNS for isentropic and heat conductive flows with some compatible conditions. With small initial energy, Huang, Li, Xin \cite{Huang.L.X2012,Huang.L2018} established the global well-posedness for the isentropic and heat-conductive flows. However, These solutions have infinite entropy in the vacuum area. In fact, as pointed out by Xin and Yan \cite{Xin.1998a,Xin.Y2013}, the classical solutions with bounded entropy to the CNS without heat conduction will blow up in finite time due to the appearance of vacuum state. Such a result also was extended to the heat-conductive CNS in \cite{Cho.J2006,Jiu.W.X2015}. Recently, Li, Wang and Xin \cite{Li.W.X2019} showed that there does not exist any local-in-time classical solution in the inhomogeneous Sobolev space to the CNS if vacuum exists in general.

In order to establish a solution to the CNS with bounded entropy and vacuum states, the studies above motivate us to consider the free boundary problem, and investigate \emph{singular boundary conditions} of the profiles across the vacuum boundary. 
When the density connects to the vacuum on the moving boundary with a jump, the local well-posedness theory and the global stability of equilibria can be tracked back to Solonnikov, Tani, Zadrzy\'nska, and Zaj\c{a}czkowski, \cite{Solon.T1990,Zadrz.2001, Zaja.1994}. 
On the other hand, when the density profile connects continuously to vacuum across the moving boundary, 
Jang, Masmoudi, Coutand, Lindblad, and Shkoller established the local well-posedness in \cite{Jang.M2009,Jang.M2015,Couta.L.S2010,Couta.S2011,Couta.S2012} for isentropic inviscid flows with physical vacuum condition, i.e., the sound speed of the fluid $c=\sqrt{{ P'(\rho)} }$ is $1/2-$Holder continuous across the vacuum boundary:
\begin{equation}\label{physical-vacuum}
	- \infty< \nabla_\vn (c^2) < 0 \quad \text{on~} \Gamma(t) . 
\end{equation}
The physical vacuum condition was first proposed by \cite{Liu.1996} when he studied the self-similar solutions to compressible Euler with damping; see \cite{Luo.Z2016,Zeng.2017,Zeng.2021} for the global smooth solutions to this model around the self-similar solutions. Condition \cref{physical-vacuum} also commonly appears in other physical models: the stationary solutions to the compressible Euler--Poisson or CNS--Poisson equations (i.e., Lane--Emden solutions) which models gaseous stars \cite{Chand.1957}, and self-similar expanding solutions to compressible Euler equations \cite{Sider.2017}; see \cite{Gu.L2012,Gu.L2016,Luo.X.Z2014,Luo.X.Z2016,Luo.X.Z2016a,Hadzi.J2018,Hadzi.J2018a,Parme.H.J2021} for the global smooth solutions to these equations. 
We also refer to the readers to \cite{Jang.M2012} for the investigations of the existence of classical solutions to the 1-D compressible Euler equations under the assumption of different behaviours of the initial density at the boundary. 
In the case of viscous flows, the extra viscosities bring more regularities to the velocity field. Consequently, the degeneracy is comparably causing fewer troubles; see \cite{Luo.X.Y2000,Liu.Y2000,Jang.2010,Guo.X2012,Li.Z2016,Zeng.2015,Liu.2018,Mei.W.X2018,Mai.L.M2019,Gui.W.W2019} for the isentropic flows. 

\vline

\emph{However, few studies of free boundary problems of non-isentropic flows are available.} 
The first author studied of the free boundary problem for non-isentropic flows by investigating the equilibria of the radiation gaseous stars in \cite{Liu.2018}, in which the degeneracy of density and temperature near the vacuum boundary are established. 
Later in \cite{Liu.Y2019}, we first proved the existence and uniqueness of the local-in-time strong solutions to the free boundary problem of the full CNS (constants $\mu, 2\mu+3\lambda, \kappa>0$). In \cite{Liu.Y2019}, we impose more general decay rates of the initial density and temperature near the vacuum boundary, i.e., the first line of \eqref{H1-1} and 
\begin{equation}\label{physical-vacuum-theta0}
 - \infty< \nabla_\vn (\theta_0) < 0 \quad \text{on~} \Gamma_0.
\end{equation}
Actually, the above condition \cref{physical-vacuum-theta0} can be automatically fulfilled for the solution after the initial time thanks to the Hopf's lemma of the degenerate parabolic type equation; see Theorem \ref{thm-3'} below. We also established a class of globally degenerate large solutions to the free boundary problem of the CNS with only constant shear viscosity and without bulk viscosity and heat conductivity ($\mu>0$, $2\mu+3\lambda, \kappa=0$) in \cite{Liu.Y2019a}. We would like to point out that the entropy of the solutions we found in \cite{Liu.Y2019,Liu.Y2019a} can be bounded, which is totally different from the blowup results of the Cauchy problem in \cite{Xin.1998a, Li.W.X2019}. Recently Chen et.al. \cite{Chen.H.W.W2021} offered different \textit{a priori} estimates in conormal Sobolev spaces for the local-in-time solutions, with fewer compatibility conditions but under the the physical vacuum condition in the isentropic case (i.e., \eqref{H1-1} with $\delta=1/(\gamma-1)$). Also see some works on other models of non-isentropic flows in \cite{Hong.L.Z2018,Geng.L.W.X2019, Mei.2020,Ricka.H.J2021,Ricka.2021,Ricka.2021a}. 

\vline

This work is the first step towards studying the dynamics of flows with \textbf{bounded entropy and degenerate, temperature-dependent transport coefficients} in the setting of vacuum free boundary problems. 
Before establishing the energy estimates and the well-posedness theorems, it is important to first study the boundary behaviors of classical solutions if they exist, and then to investigate the proper \emph{singular boundary conditions} of the profiles across the vacuum boundary. 
Inspired by the work of \cite{Li.W.X2019}, we propose some non-existence theorems to the vacuum free boundary problem of the non-isentropic CNS: 
it is proved in this paper that in the three-dimensional case and $0\leq \alpha \leq 1/(\gamma-1)$, when the heat conductivity does not vanish, the normal derivative of the temperature of the classical solution across the free boundary does not vanish, and the entropy will immediately blowup if the physical vacuum condition is not satisfied ($\delta=1/(\gamma-1)$). 
When the heat conductivity vanishes but the bulk viscosity does not, the classical solution does not exist for any short time provided that \eqref{H1-1} is satisfied and the initial velocity is expanding near the boundary (see \eqref{H2*} below). We also have a similar non-existence result for the one-dimensional case, but with more restrictions on the parameter $\delta$, $\alpha$ and the decaying condition of the initial velocity. Such non-existence theorems show that in the previous well-posedness results (\cite{Liu.Y2019,Liu.Y2019a}), singular boundary conditions across the vacuum boundary imposed on the initial data are reasonable. 

In contrast to the Cauchy problem in \cite{Li.W.X2019}, where the velocity and its derivatives of classical solutions have to vanish on the vacuum boundary, the behaviors for the velocity on the vacuum boundary are not clear for the free boundary problem. Moreover, in the case of the degenerate, temperature-dependent viscosities and heat-conductivity, the ellipticity of the momentum and temperature equations in Lagrangian coordinates may degenerate on the vacuum boundary, and in this case the maximum principles in \cite{Li.W.X2019} cannot be applied. This calls for new analysis on the boundary behaviors of the velocity and entropy, and more general maximum principles for degenerate parabolic equations.
To the best of our knowledge, this paper is the first study concerning the ill-posedness of the vacuum free boundary problem of non-isentropic CNS. 
Our result demonstrates that singular boundary conditions across the vacuum boundary in the previous well-posedness results (\cite{Liu.Y2019,Liu.Y2019a}) are reasonable.
Moreover, our result provides a first-step investigation of the well-posedness theory of the vacuum free boundary problem for non-isentropic viscous flows with degenerate, temperature-dependent transport coefficients.

\subsection{Main results}
In the rest of the paper, we use $W^{k,p}, H^m$ to denote the classical Sobolev space in $\Omega(t)$,  $C^{m_1}_{m_2}(\Omega(t)\times[0,T])$ to denote the set of functions that are $C^{m_1}$ continuous in space and $C^{m_2}$ in time in the domain $\Omega(t)\times[0,T]$.
The main results of this paper are as follows:
\begin{theorem}\label{thm-3}
	In the three-dimensional spherically symmetric case, that with positive bulk viscosity and zero heat conductivity (i.e., $2\bar{\mu}+3\bar{\lambda}>0$, $\kappa=0$), 
	assume that $0\leq \alpha\leq\frac{1}{\gamma-1},$ and $\rho_0$ satisfies \eqref{H1-1},  
	and that there exists $r_0\in (1-d_0,1 )$ such that
	\begin{equation}\label{H2*}
		\frac{y}{\abs{y}{}} \cdot \vu_0(y)>0 ~\text{for}~\abs{y}{}=r_0;\quad \frac{y}{\abs{y}{}} \cdot \vu_0(y)\geq 0 ~\text{for}~r_0<\abs{y}{}<1,
	\end{equation}
	where $ d_0 \in(0,1) $ is a small constant depending only on $ C_1 $ and $C_2 $ in \eqref{H1-1}.
	Then the vacuum free boundary problem \eqref{NS}, \eqref{boundary}, \eqref{kinematic}, \eqref{initial}  and \eqref{boundary-2} has no solution $(\rho,u,\theta)$ satisfying 
	$$
	\begin{aligned}
		&\rho, \theta\in C^{1,1}(\Omega(t) \times [0,T]) \cap C(\overline{\Omega(t)}\times [0,T]), \\
		&u\in C^2_1(\Omega(t) \times [0,T]) \cap C( \overline{\Omega(t)} \times[0,T])\\
		&\qquad \cap L^{\infty}(0,T; W^{2,\infty} (\Omega(t)) ) ,	\\
		&u_t \in L^{\infty} (0,T; L^\infty(\Omega(t))),
	\end{aligned}
	$$
	 with entropy $s$ satisfying 
	\begin{align}
		&s\in L^{\infty}(0,T; W^{1,\infty}(\Omega(t))), \quad s_t\in L^2(0,T;L^{\infty}(\Omega(t))),  \label{S-assump-3}
	\end{align} 
	for any positive time $T$.
\end{theorem} 

\begin{remark}\label{rem-regularity}  The regularity of $(\rho,u,\theta)$ in Theorem \ref{thm-3} is weaker than those in \cite{Li.W.X2019} ($ C^1([0,T];H^m(\Omega(t)))$, $m\geq3$). In particular, the derivatives of $\rho,\theta$ are not required continuous up to the boundary. 
	
When $2\mu+3\lambda=0$ and $\kappa=0$, the well-posedness theory of the classical solutions is quite different: the author (\cite{Liu.Y2019a}) proved the existence of global-in-time classical solutions with bounded entropy; moreover, the entropy of the self-similar solutions founded in \cite{Liu.Y2019a} is independent of time along the flow map, and thus if their entropy and entropy derivatives are bounded initially, they remain so in the coming future.
\end{remark}
Since  $s, s_t$ and $s_y$ are also continuous in  $\Omega(t)$, thanks to this regularity of $\rho, \theta$ and the positivity of $\rho$ in $\Omega(t)$, \eqref{S-assump-3} can be regarded as a condition describing the boundary behaviors of $s,s_t$ and $ s_x$. 

\begin{theorem}\label{thm-3'}
	In the three-dimensional spherically symmetric case with nontrivial heat conductivity (i.e., $\bar{\kappa}>0$), 
	assume that $0\leq \alpha\leq\frac{1}{\gamma-1}$.
	Then the solution $(\rho, u, \theta)$ to the vacuum free boundary problem \eqref{NS}, \eqref{boundary}, \eqref{kinematic}, \eqref{initial} and \eqref{boundary-2}  with 
	$$\rho \in C^{1,1}(\Omega(t) \times [0,T]) \cap C(\overline{\Omega(t)}\times [0,T]), $$
	$$\theta\in  C^2_1(\Omega(t) \times [0,T]) \cap C(\overline{\Omega(t)}\times [0,T]) 
	,$$
	$$u\in C^2_1(\Omega(t) \times [0,T]) \cap L^{\infty}(0,T; W^{2,\infty} (\Omega(t)) ), $$
	$$u_t \in  L^{\infty} (0,T; L^\infty(\Omega(t)))$$ for any positive time $T$, has to satisfy that 
	 \begin{align} 
	 	& \nabla_\vn (\theta) < 0 \quad \text{on~} \Gamma(t). \label{physical-vacuum-theta}
	 \end{align}
    Moreover, if initially the first line of \eqref{H1-1} is satisfied and $\theta\in L^{\infty}(0,T; H^1_0 (\Omega(t)) )$, then the entropy is bounded in space-time if and only if $\delta=1/(\gamma-1)$.
\end{theorem} 

\begin{remark}
	$\theta\in L^{\infty}(0,T; H^1_0 (\Omega(t)) )$ is the usual improved regularity for the solutions to parabolic equations. In fact, when the transport coefficients are constants and $2\mu+3\lambda>0$, $\kappa>0$,  the authors (\cite{Liu.Y2019}) proved the existence of locally-in-time solutions $(\rho,u,\theta)$ satisfying $u,\theta \in  L^{\infty}(0,T; H^3 (\Omega(t)) ),$ and $u_t,\theta_t \in  L^{\infty}(0,T; H^3 (\Omega(t)) )$, and such solutions could be classical with more regular initial data.
	
	In \cite{Liu.Y2019}, \eqref{physical-vacuum-theta} is required initially but only used to ensure $\Theta>0$ in $\Omega(t)\times[0,T]$. In this paper, we find that \eqref{physical-vacuum-theta} and $\Theta>0$ can be automatically fulfilled thanks to the maximum principle of the degenerate parabolic type equation.
\end{remark}

\eqref{physical-vacuum-theta} in Theorem \ref{thm-3'} also holds without the spherical symmetry, since the Hopf's lemma and maximum principles for the degenerate parabolic operator hold for general three-dimensional domains. We prove it in the spherical symmetry for simplicity.

In the one-dimensional case, without the special structure of the equations in three-dimensional spherically symmetric case, there has to be more restrictions on the initial data. 
\begin{theorem}\label{thm-1}
	In the one-dimensional case without heat conductivity (i.e., $\kappa=0$), 
    assume that $0\leq \alpha\leq\frac{1}{\gamma-1},$ and $\rho_0$ satisfies \eqref{H1-1}, and
	\begin{equation}\label{H3}
	\begin{gathered}
		0<\frac{1}{\delta(\gamma-1)}<\min \{ \frac{1}{2}(\frac{\gamma-2}{\gamma-1}+\alpha), 1 \}, \quad 
		\frac{ u_0}{d^{1/2}}\in L^{2}(0,1),
	\end{gathered}
	\end{equation}
	and
	\begin{equation}\label{H2}
	\begin{aligned}
	\exists ~ Y_0 \in (0,d_0 ) ~ \text{such that}~ &u_0(Y_0)<0 ~\text{and} ~ u_0(y)\leq 0 ~\text{for} ~ \forall~0<y<Y_0,  \\
	\text{or}~~~~ \exists ~ Y_0  \in (1-d_0,1) ~ \text{such that}~ & u_0(Y_0)>0 ~\text{and} ~ u_0(y)\geq 0 ~ \text{for} ~ \forall~Y_0<y<1,
	\end{aligned} 
	\end{equation}
     where $ d_0 \in(0,1) $ is a small constant depending only on $ C_1 $ and $C_2 $ in \eqref{H1-1}.
	Then the vacuum free boundary problem \eqref{NS}, \eqref{boundary}, \eqref{kinematic}, \eqref{initial}, and \eqref{boundary-2} has no solution $(\rho,u,\theta)$ satisfying 
    $$
    \begin{aligned}
    	&\rho, \theta\in C^{1,1}(\Omega(t) \times [0,T]) \cap C(\overline{\Omega(t)}\times [0,T]), \\
    	&u\in C^2_1(\Omega(t) \times [0,T]) \cap C( \overline{\Omega(t)} \times[0,T]) \\
    	&\qquad\cap L^{\infty}(0,T; W^{2,\infty} (\Omega(t)) ),	\\
    	&u_t \in L^{\infty} (0,T; L^\infty(\Omega(t)))
    \end{aligned}
    $$
	with entropy $s$ satisfying 
	\begin{align}
		&s\in L^{\infty}(0,T; L^{\infty}(\Omega(t))), \quad s_t\in L^2(0,T;L^{\infty}(\Omega(t))) \label{S-assump-1}
	\end{align} 
	for any positive time $T$.
\end{theorem}

Note that \eqref{H3} does not hold for the physical vacuum condition ($\delta=1/(\gamma-1)$). Such a vacuum free boundary problem in the one-dimensional case might be well-posed locally for initial data with proper decaying rate to the boundary.

This work is organized as follows:  in Section 2 we prepare Hopf's lemma and the strong maximum principle for the degenerate parabolic operators used in our proof. In the subsequent sections, we first consider the one-dimensional case,  formulate the vacuum free boundary problem in the Lagrangian coordinates and prove Theorem \ref{thm-1}. Theorem \ref{thm-3} and Theorem \ref{thm-3'} are proved in Sections \ref{sect-proof-thm-3} and \ref{sect-proof-thm-3'}.

\section{The Maximum principles and Hopf's lemma} \label{sect-max}

In this section, we will prepare Hopf's lemma and the strong maximum principle for the degenerate parabolic operator 
\begin{equation}\label{def-L}
	Lw := a_0(x,t) \partial_t w- a(x,t) \partial_x^2 w -b(x,t) \partial_x w -c(x,t)w,
\end{equation}
where $a_0,a,b,$ and $c$ are continuous on $[0,d_0]\times [0,T]$ (or $[1-d_0,1]\times [0,T]$, respectively), and satisfy that, for some constant $ C \in (0,\infty) $, 
\begin{equation}\label{L-coeff}
	\begin{gathered}
		0 < a,~\frac{a_0}{a} < C \quad \text{in} ~ (0,d_0)\times [0,T] \qquad (\text{or in} ~ (1-d_0,1)\times [0,T], ~ \text{respectively} ),  \\
		\text{and} \qquad
		\frac{b}{a}>-C, ~ -\frac{Ca}{d} < c\leq 0 \quad \text{in} ~ [0,d_0]\times [0,T]\\
		(\text{or} \quad  \frac{b}{a}<C, ~ -\frac{Ca}{d} < c\leq 0  \quad \text{in} ~ [1-d_0,1]\times [0,T], ~ \text{respectively}).
	\end{gathered}
\end{equation}
Here $d_0 \leq 1/4$, and $d = \min \lbrace x, 1-x \rbrace$ is the distance to $0$ and $1$ defined as in the introduction, $a_0$ and $a$ might degenerate at the boundary.
Thanks to $ 0 \leq \alpha(\gamma-1) \leq 1 $, 
the operators $\rho_0 \partial_t +\sL_1 $, $\rho_0 \partial_t +\sL_3 ,$ and  $\rho_0 \partial_t +\tilde{\sL}_3 $ in \eqref{sign-1}, \eqref{sign-3}, and \eqref{sign-3'} respectively in the subsequent sections are the form of $L$.

The idea of the proof is classical; see \cite{Fried.1964} for the classical parabolic operator, and \cite{Li.W.X2019} for the parabolic operator with degenerate coefficient in temporal derivative. 
The crucial point is to verify that the auxiliary functions satisfy the same differential inequality \eqref{sign-w} in the case that of more degenerate coefficients; see \eqref{sign-phi} for instance.

Recall that for a bounded domain $D$ in $\mathbb{R}^n \times \mathbb{R}^+$, its parabolic boundary $\partial_p D$  is the subset of the boundary $\partial D$ such that for any point $(x_0,t_0)\in  \partial_p D$,  the set $B_l(x_0) \times (t_0-l^2, t_0] \cap (\mathbb{R}^n \times \mathbb{R}^+ )\setminus \overline{D}$ is nonempty for any small $l>0$. The non-negative and non-positive parts of a function $w$ are denoted as $w^+=\max\{w,0\}$ and $w^-=-\min\{w,0\}$, respectively, such that $ w = w^+ - w^- $.
We start with the weak maximum principle.
\begin{lemma}[Weak maximum principle]\label{lem-weak}
	Suppose that $Q$ is an open domain,  $Q_T:=Q\times (0,T]$, and  $w \in C^2_1 (Q_T) \cap C(\overline{Q_T})$.
	\begin{enumerate}[label = (\roman*),ref = (\roman*)]
		\item If $w$ satisfies	
		\begin{equation} \label{sign-w-1}
			Lw < 0 \quad\text{in} ~ Q_T,
		\end{equation}
		and $ w $ attains its non-negative maximum at $ (x_0,t_0) $, then $ (x_0,t_0) \in \partial_p Q_T $.
		\item If $ w $ satisfies
		\begin{equation} \label{sign-w}
			Lw \leq 0 \quad\text{in~} Q_T,
		\end{equation}
		then 
		\begin{equation}
			\max_{\overline{Q_T}} w \leq \max_{\partial_p Q_T} w^+.
		\end{equation}
		\item \label{w-min-prcp-1} If $w$ satisfies	
		\begin{equation} \label{sign-w-2}
			Lw > 0 \quad\text{in} ~ Q_T,
		\end{equation}
		and $ w $ attains its non-positive minimum at $ (x_0,t_0) $, then $ (x_0,t_0) \in \partial_p Q_T $.
		\item \label{w-min-prcp-2} If $ w $ satisfies
		\begin{equation} \label{sign-w'}
			Lw \geq 0 \quad\text{in~} Q_T,
		\end{equation}
		then 
		\begin{equation}
			\min_{\overline{Q_T}} w \geq -\max_{\partial_p Q_T} w^-.
		\end{equation}
	\end{enumerate}
	
\end{lemma}

\begin{proof}
	%
	We focus on the proofs of \ref{w-min-prcp-1} and \ref{w-min-prcp-2}. The proofs of the rest are similar and left to readers. 
	
	To prove \ref{w-min-prcp-1}, 
	assume that $ w $ attains its non-positive minimum at some point $(x_0,t_0)$ inside $Q_T$. 
	Since $c\leq0$, we have
	$$ w(x_0,t_0) \leq 0 ,\quad w_t(x_0,t_0) \leq 0, \quad w_x(x_0,t_0)=0, \quad w_{xx}(x_0,t_0)\geq 0,$$
	and thus $L w \leq 0$ at $(x_0,t_0)$, which contradicts \eqref{sign-w-2}.
	
	To prove \ref{w-min-prcp-2}, consider $\phi(\varepsilon,x,t):=w(x,t)+\varepsilon t$ for $\varepsilon>0$, which satisfies
	\begin{equation*}
		L\phi=Lw+\varepsilon a_0 > 0 \quad\text{in~} Q_T.
	\end{equation*}
	If $w$ attains a negative minimum at some point in $Q_T$, then $\phi(\varepsilon,\cdot)$ also attains a negative minimum at some point in $Q_T$, provided $\varepsilon>0$ is small enough. This is a contradiction to the conclusion \ref{w-min-prcp-1} for $\phi$. Therefore, \ref{w-min-prcp-2} is proved.
\end{proof}

\begin{proposition}[Hopf's lemma]
	\label{prop-Hopf}
	
	\begin{enumerate}[label = (\roman*),ref = (\roman*)]
		\item Suppose that $w \in C^2_1 ((0,d_0)\times (0,T]) \cap C([0,d_0]\times [0,T])$ satisfying the inequality \eqref{sign-w} throughout $(0,d_0)\times (0,T]$, and a point $(0,t_0)$ satisfying $w(0,t_0)\geq0$ and $w(0,t_0)>w(x,t)$ for any $(x,t)$ in the neighborhood $D$ of $(0,t_0)$:
		\begin{align*}
			&D:=\{(x,t);(x-l)^2+(t_0-t)<l^2,~0<x<\frac{l}{2},~0<t \leq t_0\}, \\
			&\qquad\qquad \text{where~} 0<l<d_0,~ t_0-\frac{3l^2}{4}>0,
		\end{align*}
		it holds that 
		$$ 
		\partial_x w(0,t_0)<0.
		$$
		
		\item \label{hopf-2} Suppose that $w \in C^2_1 ((1-d_0,1)\times (0,T]) \cap C([1-d_0,1]\times [0,T])$ satisfying the inequality \eqref{sign-w'} throughout $(1-d_0,1)\times (0,T]$, and a point $(1,t_0)$ satisfying $w(1,t_0)\leq0$ and $w(x,t)>w(1,t_0)$ for any $(x,t)$ in the neighborhood $D$ of $(1,t_0)$:
		\begin{align*}
			&D:=\{(x,t);(x-1+l)^2+(t_0-t)<l^2,~1-\frac{l}{2}<x<1,~0<t \leq t_0\}, \\
			&\qquad\qquad \text{where~} 0<l<d_0,~ t_0-\frac{3l^2}{4}>0,
		\end{align*}
		it holds that 
		$$
		\partial_x w(1,t_0) < 0.
		$$
	\end{enumerate}
\end{proposition}
\begin{proof}
	It suffices to prove \ref{hopf-2}. For $(x,t)\in D$, we define the auxiliary functions 
	\begin{gather*}
		q(\beta,x,t)=e^{-\beta[(x-1+l)^2+(t_0-t)] }-e^{-\beta l^2},\\
		\phi(\varepsilon,\beta,x,t) =w(x,t)- w(1,t_0) -\varepsilon q(\beta,x,t),
	\end{gather*}
	where $\beta$ and $\varepsilon$ are constants to be determined. 
	
	On the boundary $\Sigma_1$ defined by (away from the point $(0,t_0)$)
	$$\Sigma_1:=\{(x,t); (x-1+l)^2+(t_0-t)<l^2,~x=1-\frac{l}{2},~0\leq t < t_0\},$$
	there exists $ \varepsilon_0 \in (0,1) $ small enough such that
	$w(x,t)- w(1,t_0)>\varepsilon_0$. Since $0\leq q\leq 1$ in $D$, $\phi(\varepsilon_0,\beta,x,t)>0$ on $\Sigma_1$.
	
	While on the boundary $ \Sigma_2 $ defined by
	$$\Sigma_2:=\{(x,t); (x-1+l)^2+(t_0-t)=l^2,~1-\frac{l}{2}\leq x\leq1,~0\leq t < t_0\},$$
	$q=0$, and thus $\phi(\varepsilon_0,\beta,x,t)\geq 0$ on $\Sigma_2$ due to the property of $w(1,t_0)$.
	
	Now we check the sign of $L \phi$. Since $w$ satisfies \eqref{sign-w'}, we have
	$$
	L \phi(\varepsilon_0,\beta,x,t) \geq c w_0(1,t_0)-\varepsilon_0 L q(\beta,x,t)
	\geq -\varepsilon_0 L q(\beta,x,t).$$ 
	Moreover, direct calculation yields that
	\begin{align*}
			-ce^{\beta[(x-1+l)^2+(t_0-t)] }q(\beta,x,t)&= -c (1 - e^{\beta[(x-1+l)^2+(t_0-t)] - \beta l^2}) \\
			& \leq -c (-1) (\beta[(x-1+l)^2+(t_0-t)] - \beta l^2)\\
			& \leq -c \beta (l^2 - (x-1+l)^2)
			\leq - 2 \beta l c d \\
			& \leq 2 \beta l C a, \quad \text{since~} -\frac{Ca}{d} < c\leq 0,
	\end{align*}  
    and therefore, one has
    \begin{align*} &e^{\beta[(x-1+l)^2+(t_0-t)]} L q(\beta,x,t)\\
    	= & -4a(x-1+l)^2\beta^2 + a[\frac{a_0}{a} +2+2\frac{b}{a}(x-1+l)]\beta \\
    	&  -ce^{\beta[(x-1+l)^2+(t_0-t)]}q\\
    	\leq &a(-C^{-1} \beta^2 + C \beta + 2\beta l C + C).
    \end{align*}
	Therefore, when $\beta_0$ is large enough, 
	\begin{equation}\label{sign-phi}
		L \phi(\varepsilon_0,\beta_0,x,t)
		\geq -\varepsilon_0 e^{-\beta_0[(x-1+l)^2+(t_0-t)]} a(-C^{-1} \beta_0^2 + C \beta_0 +C)>0. 
	\end{equation}
	
	To summarize, with small $\varepsilon_0$ and large $\beta_0$, one has 
	\begin{equation}
		\begin{cases}
			L\phi(\varepsilon_0,\beta_0,x,t)
			>0  & \text{in~} D,\\
			\phi(\varepsilon_0,\beta_0,x,t)\geq0 &\text{on~} \partial_{p} D.
		\end{cases}
	\end{equation}
	It follows from \ref{w-min-prcp-2} in Lemma \ref{lem-weak} that 
	$$\phi(\varepsilon_0,\beta_0,x,t)\geq0 \quad \text{in~} D.$$
	Consequently, since $ \phi(\varepsilon_0,\beta_0,1,t) = 0 $, one can conclude that $ 
	\partial_x \phi(\varepsilon_0,\beta_0,1,t_0)
	\leq 0 $,  and 
	$$
	\partial_x w(1,t_0) \leq \varepsilon_0 \partial_x q(\beta_0,1,t_0) < 0.
	$$
\end{proof}

We remark that the boundedness of $\frac{a_0}{a}$, i.e., the degenerate rate of the spatial derivative coefficient is smaller than that of the temporal derivative coefficient, ensures that the maximum principles and Hopf's lemma work.
The proofs of the strong maximum principle from Hopf's lemma are routine, and thus they are omitted here. We refer interested readers to \cite{Fried.1964,Li.W.X2019} for details.

\begin{proposition}[Strong maximum principle]
	\label{prop-strong} There exist small enough $d_0$ such that
	
	(i)  for any $w \in C^2_1 ((0,d_0)\times (0,T]) \cap C([0,d_0]\times [0,T])$ satisfying the inequality \eqref{sign-w} throughout $(0,d_0)\times (0,T]$, if $w$ attains its non-negative maximum at some interior point $(x_0,t_0)$ of $(0,d_0)\times (0,T]$, then $w\equiv w(x_0,t_0)$ in $(0,d_0)\times (0,T]$;
	
	(ii)  for any $w \in C^2_1 ((1-d_0,1)\times (0,T]) \cap C([1-d_0,1]\times [0,T])$ satisfying the inequality \eqref{sign-w'} throughout $(1-d_0,1)\times (0,T]$, if $w$ attains its non-positive minimum at some interior point $(x_0,t_0)$ of $(1-d_0,1)\times (0,T]$, then $w\equiv w(x_0,t_0)$ in $(1-d_0,1)\times (0,T]$.
\end{proposition}

\begin{corollary} \label{cor-strong}
	Relaxing the constraints of $ c $ in \eqref{L-coeff} to $ -\frac{Ca}{d} < c\leq Ca_0$,
	the conclusions of \cref{prop-strong} continue to hold provided that the maximum of $w$ for (i) (or the minimum of $w$ for (ii)) is exactly $0$. 
\end{corollary}
\begin{proof}
	It suffices to prove (ii). Now we consider
		$$\phi:=e^{-Ct}w, \quad \tilde{L}\phi:=L\phi+Ca_0 \phi=e^{-Ct} L\omega \geq0 \quad \text{in~} (1-d_0,1)\times (0,T],$$
		Therefore, if $w$ attains its minimum $0$ at some interior point $(x_0,t_0)$ of $(1-d_0,1)\times (0,T]$, $\phi$ attains its minimum $0$ at the same interior point.
		In addition, thanks to the assumption that $ -\frac{Ca}{d} < c\leq Ca_0 $, $\tilde{L}$ satisfies the condition of coefficients \cref{L-coeff}. Therefore, it follows from \cref{prop-strong} that $\phi\equiv0$ in $(1-d_0,1)\times (0,T]$, and thus so does $w$.
\end{proof}

\section{Proof of Theorem \ref{thm-1}}


\subsection{Lagrangian formulation}
In the case of $ n = 1 $ and $\kappa=0$, we write $ \vu = u $ and $\Omega(t) = (\omega_1(t), \omega_2(t))$ with $\Gamma(t) = \lbrace y=\omega_1(t), y=\omega_2(t) \rbrace $. Then \eqref{kinematic} is reduced to 
\begin{equation}
	\label{kinematic-1}
	\omega_1'(t)=u(\omega_1(t),t), ~\omega_2'(t)=u(\omega_2(t),t), 
\end{equation}
and  the vacuum free boundary problem with \eqref{NS}, \eqref{boundary}, \eqref{kinematic}, \eqref{initial},  and \eqref{boundary-2} is reduced to
\begin{equation}\label{NS-1}
	\begin{cases}
		\rho_t+(\rho u)_y=0 & \text{in}\ \Omega(t), \\
		\rho u_t +\rho uu_y+p_y=(\nu u_y)_y  & \text{in}\ \Omega(t),\\
		c_v (\rho \theta_t+\rho u\theta_y)+pu_y=\nu (u_y)^2  & \text{in}\ \Omega(t),\\
		\rho>0,~\theta>0  & \text{in}\ \Omega(t),\\
		\rho=\theta=0, ~ \nu u_y= 0  &  \text{on} \ \Gamma(t),\\
		(\rho,u,\theta)=(\rho_0,u_0,\theta_0) &  \text{on} \ \Omega_0 = (0,1).
	\end{cases}
\end{equation}
Here $\nu=\bar{\nu}\theta^{\alpha},~\bar{\nu}=2\bar{\mu}+\bar{\lambda}>0$. Note that when $\alpha>0$, $\nu u_y\vert_{\Gamma(t)}=0$ thanks to $\theta\vert_{\Gamma(t)}=0$. 

For $ x \in \Omega_0 = (0,1) $, 
define the flow map $y=\eta(x,t)$ by the following ODE:
\begin{equation}\label{def-eta}
	\begin{cases}
		\eta_t(x,t)=u \circ \eta (x,t) = u(\eta(x,t),t), \\
		\eta(x,t=0)=\eta_0(x) = x.
	\end{cases}
\end{equation}
Notice that, $ \omega_1(t) = \eta(0,t) $ and $ \omega_2(t) = \eta(1,t) $ thanks to \eqref{kinematic-1}.
Then the Lagrangian unknowns are denoted as:
\begin{equation} \label{Lag-unknowns}
	f:=\rho \circ \eta, \quad
	v:=u\circ \eta, \quad
	\Theta:=\theta \circ \eta, \quad
	S := s \circ \eta.
\end{equation}
\eqref{state} implies that
\begin{equation}\label{entropy-0}
	S= c_v \bigl( \log ( \dfrac{\Theta}{f^{\gamma-1}}) + \log (\dfrac{R}{\bar A}) \bigr). 
\end{equation}
Consequently, system \eqref{NS-1} can be written as
\begin{equation}
	\label{NS-1-mid}
	\begin{cases}
		f_t+f \frac{v_x}{\eta_x}=0 & \text{in}\ (0,1)\times(0,T],\\
		f v_t+\frac{1}{\eta_x} (Rf\Theta)_x = \frac{1}{\eta_x} (\nu \frac{v_x}{\eta_x})_x & \text{in}\ (0,1)\times(0,T],\\
		c_v f\Theta_t +Rf\Theta \ \frac{v_x}{\eta_x}= \nu(\frac{v_x}{\eta_x})^2 
		&  \text{in}\ (0,1)\times(0,T],\\
		f>0,~ \Theta>0 & \text{in}\ (0,1)\times(0,T],\\
		f= \Theta = 0,\ \nu \frac{v_x}{\eta_x}=0, & \text{on} \  \{0,1\} \times (0,T],\\
		(f,v,\Theta)=(\rho_0,u_0,\theta_0) & \text{on}\ (0,1)\times\{t=0\}.
	\end{cases}
\end{equation}

Notice that $ v_x = \partial_t \eta_x $. $ f $ can be solved from \subeqref{NS-1-mid}{1} as 
\begin{equation}\label{solving-density}
	f = \rho_0/\eta_x.
\end{equation}
Moreover, we define 
\begin{equation} 
	\label{def-zeta}
	\zeta:=\frac{\bar{A}}{R}e^{S(\gamma-1)/R}=\rho_0^{-(\gamma-1) }\eta_x^{\gamma-1} \Theta,~ \zeta_0 = \frac{\theta_0}{\rho_0^{\gamma-1}},
\end{equation} 
and then an alternative form of \eqref{NS-1-mid} is, 
\begin{equation}\label{NS-1-Lag}
	\begin{cases}
		\rho_0 v_t+ (\frac{R\rho_0^{\gamma} \zeta}{\eta_x^{\gamma}})_x = (\nu \frac{v_x}{\eta_x})_x  &\text{in}\ (0,1)\times(0,T],\\
		c_v \rho_0^{\gamma}\zeta_t = \nu v_x^2 \eta_x^{\gamma-2} &\text{in}\ (0,1)\times(0,T],\\
		\nu v_x=0, ~ \rho_0 = 0  & \text{at}\ x=0,1,\\
		(v,\zeta)=(u_0,\zeta_0 )  & \text{on}\ (0,1)\times\{t=0 \}.
	\end{cases}
\end{equation}
Therefore, to prove Theorem \ref{thm-1}, it suffices to prove the following theorem in the Lagrangian coordinates:

\begin{theorem}\label{thm-1-Lag}
	When $\kappa=0$, the free boundary problem in the Lagrangian coordinates \eqref{NS-1-Lag} has no classical solution $(v,\zeta)$ satisfying 
	\begin{equation}  \label{S-assump-1-Lag}
		\left\{
		\begin{aligned}
			&v\in C^2_1((0,1) \times [0,T]) \cap C( [0,1] \times[0,T]) \\
			& \qquad \cap L^{\infty}(0,T; W^{2,\infty} (0,1) ) ,\\
			&v_t \in  L^{\infty} (0,T; L^\infty(0,1)),\\
			&\zeta\in C^{1,1}((0,1)\times [0,T]), ~\zeta, \frac{1}{\zeta}, \zeta_x\in L^{\infty}(0,T; L^{\infty}(0,1)) \\
			& \zeta_t\in L^2(0,T;L^{\infty}(0,1)),
		\end{aligned}
		\right.
	\end{equation}
	for any positive time $T$, 
	provided that \eqref{H1-1}, \eqref{H3} and \eqref{H2} hold.
\end{theorem}

Throughout this section,  we use $C$ to represent the generic positive constant, and ``$f\sim g$'' to represent ``$\frac{1}{C}\abs{g}{} \leq \abs{f}{}\leq C \abs{g}{}$ '' for some constant $C>0$. ``$f=\mathcal{O}(g)$'' means $\abs{f}{}\leq C g$ for some constant $C>0$ as $x$ or $y\rightarrow0^+$ or $1^-$. 

\vline

\subsection{Boundary behaviour of $v_x $ and $v$}\label{sect-v}
In this subsection, we will prove
\begin{proposition}\label{prop-v}
	Suppose that \eqref{H1-1} and \eqref{H3} hold true, and that $(v,\zeta)$ is the classical solution to \eqref{NS-1-Lag} with \eqref{S-assump-1-Lag}. 
	Then $v_x\vert_{x=0,1}=v\vert_{x=0,1}=0 $ for a.e. $t\in(0,T]$.
\end{proposition}

\begin{proof} Without loss of generality, 
	thanks to the fact that $v\in  L^{\infty}(0,T; W^{2,\infty} (0,1) )$ and $(\eta_0)_x=1$, we consider $T > 0 $ small enough such that 
	\begin{equation} \label{eta-bdd}
		\frac{1}{2}<\eta_x<2\quad \forall~(x,t)\in[0,1]\times[0,T].
	\end{equation}
	Then  \subeqref{NS-1-Lag}{2}, \eqref{S-assump-1-Lag}, and \eqref{eta-bdd} imply that
	$ \zeta_t \sim \nu \rho_0^{-\gamma} v_x^2  \in L^{2}([0,T], L^\infty(0,1))$. Therefore, 
	\begin{equation} \label{decay-vx}
		v_x=\mathcal{O}(\nu^{-\frac{1}{2}}\rho_0^{\frac{\gamma}{2}})=\mathcal{O}(\zeta^{-\frac{\alpha}{2}}\rho_0^{\frac{\gamma-\alpha(\gamma-1)}{2}})=\mathcal{O}(d^{\delta(\gamma-\alpha(\gamma-1))/2} ),~~\text{a.e.}~t\in[0,T],
	\end{equation}
	which yields $v_x(0,t)=v_x(1,t)=0$ provided that $\alpha<\gamma/(\gamma-1)$.

	On the other hand,
	to show $v(0,t)=v(1,t)=0$, it suffices to show that \emph{$\int_0^1\frac{v^2}{d} dx <\infty$ holds for a.e. $t\in[0,T]$}.
	Multiplying \subeqref{NS-1-Lag}{1} by $\rho_0^{-1} d^{-1}  v$ and integrating the resultant equation on $[0,1]$ yields 
	\begin{equation}
		\label{eq-v}
		\frac{1}{2} \frac{d}{dt}\int_0^1  \frac{ v^2}{d } dx 
		+ \int_0^1 \frac{v}{\rho_0 d} (\frac{R \rho_0^{\gamma} \zeta }{\eta_x^{\gamma}} )_x dx
		-\int_0^1\frac{ v}{\rho_0d } ( \nu\frac{v_x}{\eta_x} )_x dx=0.
	\end{equation}
	We estimates the last two integrals of the above equation one by one. First, direct calculation yields 
	\begin{align*}
		\abs{\int_0^1 \frac{v}{\rho_0 d} (\frac{R \rho_0^{\gamma} \zeta }{\eta_x^{\gamma}} )_x dx}{} 
		&\leq C\int_0^1 \left(
		\abs{\frac{\rho_0^{\gamma-1}\zeta_x v}{d}}{} 
		+ \abs{ \frac{\rho_0^{\gamma-2}\rho_0'\zeta v}{d}}{}
		+ \abs{\frac{\rho_0^{\gamma-1}\zeta v}{d} }{}
		\right)dy\\
		&\leq C\int_0^1 \abs{d^{\delta(\gamma-1)-2} v}{} dx \times ( \Vert \zeta_x\Vert_{L^\infty} + \Vert \zeta \Vert_{L^\infty}  ) \\
		&\leq C \int_0^1 \frac{v^2}{d} dx +  C\int_0^1d^{2\delta(\gamma-1)-3} dx \times ( \Vert \zeta_x\Vert_{L^\infty}^2 + \Vert \zeta \Vert_{L^\infty}^2  ),
	\end{align*}
	where we have applied the fact that
	\begin{equation}\label{bdd-001}
		\eta_{xx}\in L^\infty ((0,1)\times[0,T])
	\end{equation} 
	thanks to \eqref{def-eta} and \eqref{S-assump-1-Lag}. Notice that $ \Vert \zeta_x\Vert_{L^\infty}^2 + \Vert \zeta \Vert_{L^\infty}^2 $ is integrable in time thanks to \eqref{S-assump-1-Lag}.
	The above integral is integrable if 
	$$
	2\delta(\gamma-1)-3>-1, \quad \text{or equivalently} \quad \delta(\gamma-1)>1.
	$$
	Meanwhile, after applying integration by parts, one has
	\begin{equation*}
		-\int_0^1\frac{ v}{\rho_0d } ( \nu\frac{v_x}{\eta_x} )_x dx=
		\underbrace{-\frac{\nu v v_x}{\rho_0 d \eta_x} \Big\vert_{0}^{1}}_{\textbf{I}}
		+ \underbrace{\int_0^1 \frac{\nu v_x^2}{\rho_0 d \eta_x } dx}_{\textbf{II}}
		+ \underbrace{\int_0^1 (\frac{1}{ \rho_0 d } )_x \frac{\nu vv_x}{\eta_x}  dx}_{\textbf{III}}, 	
	\end{equation*}
	where
	\begin{align*} 
		&\abs{\textbf{I}}{}\leq C \lim_{x\rightarrow 0,1}d^{\delta (\frac{\gamma+\alpha (\gamma-1) }{2}-1)-1} 
		\quad  \text{thanks to \eqref{decay-vx}}, \quad \textbf{II} \geq0, \\
		&\abs{\textbf{III}}{} \leq C\int_0^1 \abs{d^{\delta (\frac{\gamma+\alpha (\gamma-1) }{2}-1)-2} v}{} dy \quad  \text{thanks to \eqref{H1-1}  and \eqref{decay-vx}}\\
		& \quad \quad 
		\leq  \int_0^1 \frac{v^2}{d} dy +  C\int_0^1 d^{\delta ((\gamma+\alpha (\gamma-1)) -2)-3}  dy.
	\end{align*}
	Consequently, $\textbf{I}=0$ and \textbf{III} is integrable provided that
	$$
	\begin{cases}
		\delta (\frac{\gamma+\alpha (\gamma-1) }{2}-1)-1>0,\\
		\delta ((\gamma+\alpha (\gamma-1)) -2)-3>-1,
	\end{cases}
	~ \text{or, equivalently} \quad  \gamma+\alpha (\gamma-1)>2+\frac{2}{\delta}.
	$$
	As a consequence, with \eqref{H3}, \eqref{eq-v} implies that
	$$	\frac{d}{dt}\int_0^1  \frac{ v^2}{d } dx \leq C\int_0^1  \frac{ v^2}{d } dx +C ( \Vert \zeta_x\Vert_{L^\infty}^2 + \Vert \zeta \Vert_{L^\infty}^2  ),$$
	and it follows from Gronwall's inequality that $ v d^{-1/2}\in L^\infty([0,T], L^{2}(0,1))$, which concludes that $v\vert_{x=0,1}=0$ for a.e. $t\in[0,T]$. Proposition \ref{prop-v} is proved.
\end{proof}

\subsection{Proof of Theorem \ref{thm-1-Lag} by contradiction}\label{sect-con}
\begin{proof}[Proof of Theorem \ref{thm-1-Lag}]
	Suppose that there exists a classical solution $(v,\zeta)$ to \eqref{NS-1-Lag} with \eqref{S-assump-1-Lag}, \eqref{H1-1}, \eqref{H3}, and \eqref{H2} for some time $T > 0$. Without loss of generality, we only consider in \eqref{H2} that $u_0(Y_0)<0$ and $u_0(y)\leq 0, ~\forall~0<y<Y_0$. The other scenario follows similarly. 
	
	Thanks to $v\in C^2_1((0,1) \times [0,T]) \cap C( [0,1] \times[0,T])$,	it follows that  $v(Y_0,t)<0$ for $t\in[0,T']$ with some small $0 < T' < T$.
	At the same time, \eqref{H1-1}, \eqref{S-assump-1-Lag}, and \eqref{bdd-001} imply that
	$$ -(\frac{R\rho_0^{\gamma} \zeta}{\eta_x^{\gamma}})_x = -\frac{R\rho_0^{\gamma} \zeta }{\eta_x^{\gamma}} 
	(\frac{\gamma(\rho_0)_x}{\rho_0}+\frac{\zeta_x}{\zeta}-\frac{\gamma \eta_{xx}}{\eta_x}) = - d^{\gamma\delta}\mathcal{O}(\frac{1}{d}) <0  \quad \text{in}~(0,d_0)\times (0,T'], $$ 
	for small enough $d_0$ depending on $ C_1 $ and $C_2 $ in \eqref{H1-1}.
	Thus, \subeqref{NS-1-Lag}{1} implies that $v$ satisfies the differential inequality
	\begin{equation} \label{sign-1}
		\begin{aligned}
			(\rho_0 \partial_t +\sL_1)v:=& (\rho_0 \partial_t - \frac{\nu}{\eta_x} \partial_{x}^2 -(\frac{\nu}{\eta_x})_x \partial_{x})v \\
			=& -(\frac{R\rho_0^{\gamma} \zeta}{\eta_x^{\gamma}})_x < 0 \quad \text{in}~(0,d_0)\times (0,T'].
		\end{aligned}
	\end{equation}
    $\rho_0 \partial_t +\sL_1$ satisfies \eqref{L-coeff} for $0\leq \alpha(\gamma-1)\leq1$ with
    \begin{equation}\label{L1-coeff}
    	\begin{aligned}
    		& a_0=\rho_0 ,~ a=\frac{\nu }{\eta_x},~b=(\frac{\nu}{\eta_x})_x,~c\equiv0,\\
    		&\frac{a_0}{a}=\mathcal{O}(\rho_0^{1-\alpha(\gamma-1)}) <C,\\
    		& \frac{b}{a}=\frac{\eta_x}{\nu} (\frac{\nu}{\eta_x})_x = \frac{ \eta_x^{\alpha (\gamma-1)+1} }{\rho_0^{\alpha (\gamma-1) } \zeta^{\alpha} } ( \frac{ \rho_0^{\alpha (\gamma-1) } \zeta^{\alpha}  }{\eta_x^{\alpha (\gamma-1)+1}} )_x \\ 
    		& \quad =  \Big( \frac{(\rho_0^{\alpha (\gamma-1) } )_x}{\rho_0^{\alpha (\gamma-1) } }  +   \frac{ \eta_x^{\alpha (\gamma-1)+1} }{ \zeta^{\alpha} }(\frac{  \zeta^{\alpha}  }{\eta_x^{\alpha (\gamma-1)+1}} )_x\Big),\\
    		&\quad >\mathcal{O}(\frac{1}{d})-C \quad \text{in} ~ (0,d_0)\times [0,T'] \\ 
    		(\text{or}	&  \quad < -\mathcal{O}(\frac{1}{d})+C \quad \text{in} ~ (1-d_0,1)\times [0,T'], ~ \text{respectively}),
    	\end{aligned}
    \end{equation}
    thanks to \eqref{H1-1} and \eqref{S-assump-1-Lag}.
    We only need to consider the case when $b$ is continuous; otherwise, since $ (\frac{\nu}{\eta_x})_x  \sim d^{\delta\alpha(\gamma-1)-1}$, we consider the operator $d^{1-\delta\alpha(\gamma-1)} (\rho_0 \partial_t + \sL_1 )$, i.e., setting 
    \begin{equation}
    	\label{L1-coff-re}
    	a_0=\rho_0 d^{1-\delta\alpha(\gamma-1)},~ a=\frac{\nu }{\eta_x}d^{1-\delta\alpha(\gamma-1)},~b=(\frac{\nu}{\eta_x})_xd^{1-\delta\alpha(\gamma-1)},~c\equiv0.
    \end{equation}  
    Therefore, \eqref{L-coeff} is satisfied for $\rho_0 \partial_t + \sL_1$ provided that $ d_0 $ is small enough. 

	Thanks to $C^2_1((0,1) \times [0,T]) \cap C( [0,1] \times[0,T])$, \eqref{H2}, \eqref{sign-1}, and $v(0,t)=0$ from Proposition \ref{prop-v}, it follows from the strong maximum principle Proposition \ref{prop-strong} that  $v<0$ in $(0,Y_0)\times[0,T']$.
	However, it follows from Hopf's lemma, i.e., Proposition, \ref{prop-Hopf} that $v_x(0,t)<0$ for $t\in[0,T']$, which contradicts to Proposition \ref{prop-v}. 
	Consequently, Theorem \ref{thm-1-Lag} is proved.
\end{proof}

\section{Proof of Theorem \ref{thm-3}} \label{sect-proof-thm-3}

\subsection{Lagrangian formulation} In the case of $ n = 3 $ with spherical symmetry, we write $r=\abs{y}{}$, $\vu(y)=u(r)\frac{y}{r}$,  $\Omega(t) = B((0,0,0),\bar{r}(t))$,  and $\Gamma(t) = \lbrace r=\bar{r}(t) \rbrace $. Then \eqref{kinematic} is reduced to 
\begin{equation}
	\label{kinematic-3}
	\bar{r}'(t)=u(\bar{r},t),
\end{equation}
and the vacuum free boundary problem with \eqref{NS}, \eqref{boundary}, \eqref{kinematic}, \eqref{initial},  and \eqref{boundary-2} is reduced to
\begin{equation}\label{NS-3}
	\begin{cases}
		(r^2 \rho)_t+(r^2 \rho u)_r=0 & \text{for}\ 0 <r<\bar{r}(t), \\
		(r^2 \rho u)_t+(r^2 \rho u^2)_r+r^2 p_r= (2\mu +\lambda) r^2 \big( \frac{(r^2u)_r}{r^2}\big)_r \\
		\qquad\qquad + r^2(2\mu_r u_r+\lambda_r \frac{(r^2u)_r}{r^2}) & \text{for}\ 0 < r<\bar{r}(t),\\
		c_v \big( (r^2 \rho \theta)_t + (r^2\rho u\theta )_r  \big) + R\rho\theta (r^2u)_r =\frac{4}{3} \mu r^2 (u_r -\frac{u}{r})^2 &\\
		\qquad\qquad +(\frac{2}{3}\mu +\lambda) r^2 (u_r +2\frac{u}{r})^2+ \kappa ( r^2\theta_r )_r   & \text{for}\ 0 <r<\bar{r}(t),\\
		\rho>0,~\theta>0  & \text{for}\ r<\bar{r}(t),\\
		\rho=\theta=0,~~ 2\mu u_r  +\lambda (u_r +2\frac{u}{r})=0  & \text{for}\ r=\bar{r}(t),\\
		(\rho,u,\theta)=(\rho_0,u_0,\theta_0) &  \text{on} \ \Omega_0,
	\end{cases}
\end{equation}

Similarly as in the one-dimensional case, we define the flow map $r=\eta(x,t)$ and the Lagrangian unknowns as in \eqref{def-eta} and \eqref{Lag-unknowns}, respectively. Note that now
$f=x^2 \rho_0 / (\eta^2 \eta_x).$

When $k=0$,  we define $\zeta$ by 
\begin{equation} 
	\label{def-zeta*}
	\zeta:=\frac{\bar{A}}{R}e^{S(\gamma-1)/R}= (\frac{x^2\rho_0}{\eta^2\eta_x}  )^{-(\gamma-1)} \Theta,~\text{with~}  \zeta \vert_{t=0}\zeta_0 = \frac{\theta_0}{\rho_0^{\gamma-1}}.
\end{equation} 
In the Lagrangian coordinates, system \eqref{NS-3} can be written as 
\begin{equation}\label{NS-3-Lag}
	\begin{cases}
		\frac{x^2\rho_0}{\eta^2} v_t+ (\frac{Rx^{2\gamma} \rho_0^{\gamma} \zeta}{\eta^{2\gamma} \eta_x^{\gamma} })_x =(2\mu+\lambda)(\frac{v_x}{\eta_x}+2\frac{v}{\eta})_x + (2\mu_x\frac{ v_x}{\eta_x} + \lambda_x (\frac{v_x}{\eta_x}+2\frac{v}{\eta})  )   &\text{in}\ [0,1)\times(0,T],\\
		c_v \frac{x^{2\gamma} \rho_0^{\gamma} }{\eta^{2\gamma} \eta_x^{\gamma} } \zeta_t= \frac{4}{3} \mu (\frac{v_x}{\eta_x} -\frac{v}{\eta})^2 + (\frac{2}{3}\mu +\lambda)  (\frac{v_x}{\eta_x} +2\frac{v}{\eta} )^2 &\text{in}\ [0,1)\times(0,T],\\
		2 \mu \frac{v_x}{\eta_x} + \lambda (\frac{v_x}{\eta_x} +2\frac{v}{\eta})=0 & \text{at}\ x=1,\\
		(v,\zeta)=(u_0,\zeta_0)  & \text{on}\ [0,1)\times\{t=0 \}.
	\end{cases}
\end{equation}

To prove Theorem \ref{thm-3}, it suffices to prove the following theorem in Lagrangian coordinates:
\begin{theorem}\label{thm-3-Lag}
	When $2\bar{\mu}+3\bar{\lambda}>0$, $\kappa=0$, $0\leq \alpha\leq\frac{1}{\gamma-1},$ the free boundary problem in Lagrangian coordinates \eqref{NS-3-Lag} has no solution $(v,\zeta)$ for any positive time $T$ satisfying,
	\begin{equation}
		\label{S-assump-3-Lag}
		\left\{
		\begin{aligned} 
			&v\in C^2_1([0,1) \times [0,T]) \cap C( [0,1] \times[0,T]) \\
			&\qquad\cap L^{\infty}(0,T; W^{2,\infty} (0,1) ), \\
			&v_t \in L^{\infty} (0,T; L^\infty(0,1)),\\
			&\zeta\in C^{1,1} ([0,1)\times[0,T]), ~  \zeta, \frac{1}{\zeta}, \zeta_x\in L^{\infty}(0,T; L^{\infty}(0,1))\\
			&\zeta_t\in L^2(0,T;L^{\infty}(0,1)),
		\end{aligned}
		\right.
	\end{equation}
	provided that initially \eqref{H1-1} and \eqref{H2*} hold.
\end{theorem}

Similarly as before, we use ``$f\sim g$'' to represent ``$\frac{1}{C}\abs{g}{} \leq \abs{f}{}\leq C \abs{g}{}$ '' for some constant $C>0$. ``$f=\mathcal{O}(g)$'' means $\abs{f}{}\leq C g$ as $x \rightarrow 1^-$. 

\vline

\subsection{Boundary behaviors of $v_x,v$}\label{sect-v-3}
In this subsection, we will prove
\begin{proposition}\label{prop-v-3}
	Suppose that $(v,\zeta)$ is the classical solution to \eqref{NS-1-Lag} satisfying \eqref{S-assump-3-Lag} and \eqref{H1-1}. 
	Then $v_x(1,t)=v(1,t)=0$ for a.e. $t\in[0,T]$.
\end{proposition}

\begin{proof}
	Since $v\in  L^{\infty}(0,T; W^{2,\infty} (0,1) ) $ and $(\eta_0)_x=1$, then for small enough $T$, we have 
	\begin{equation} \label{eta-bdd-3}
		\frac{1}{2}<\eta_x<2\quad \forall~(x,t)\in[0,1]\times[0,T].
	\end{equation}
	It follows from the above, $\zeta_t\in L^{2}([0,T], L^\infty[0,1) )$ (from \eqref{S-assump-1-Lag}) and \subeqref{NS-3-Lag}{2} that 
	$$ \rho_0^{-\gamma} \{ \frac{4}{3} \mu (\frac{v_x}{\eta_x} -\frac{v}{\eta})^2 + (\frac{2}{3}\mu +\lambda)  (\frac{v_x}{\eta_x} +2\frac{v}{\eta} )^2\} \in L^{2}([0,T], L^\infty[0,1)). $$
	Since $\mu>0, 2\mu+3\lambda>0$,
	\begin{equation} \label{decay-vx-3}
		\frac{v_x}{\eta_x} -\frac{v}{\eta}, \frac{v_x}{\eta_x} +2\frac{v}{\eta}  =\mathcal{O}(\zeta^{-\frac{\alpha}{2}}\rho_0^{\frac{\gamma-\alpha(\gamma-1)}{2}})=\mathcal{O}(d^{\delta(\gamma-\alpha(\gamma-1))/2} ),~~\text{a.e.}~t\in[0,T],
	\end{equation}
	Therefore, $v_x(1,t)=v(1,t)=0$ provided that $\alpha<\gamma/(\gamma-1)$.
\end{proof}

\subsection{Proof of Theorem \ref{thm-3-Lag} by contradiction}\label{sect-con-3}
\begin{proof}[Proof of Theorem \ref{thm-3-Lag}]
	Suppose that there exists a classical solution $(v,\zeta)$ to \eqref{NS-3-Lag} with \eqref{S-assump-3-Lag}, \eqref{H1-1}, and \eqref{H2*} for some time $T$.
	
	The fact that $v\in C^2_1([0,1) \times [0,T]) \cap C( [0,1] \times[0,T])$ implies that, thanks to \eqref{H2*}, $v_0(r)>0$ for $\abs{y}{}=r_0$ for $t\in [0,T']$ with some small $0<T'<T$. In particular, we still have \eqref{bdd-001} with the spatial domain replaced by $ (1-d_0,1) $. 
	At the same time, \eqref{H1-1}, \eqref{bdd-001}, and \eqref{S-assump-3-Lag} imply that, for a.e. $ t \in (0,T) $,
	$$ -(\frac{Rx^{2\gamma} \rho_0^{\gamma} \zeta}{\eta^{2\gamma} \eta_x^{\gamma} })_x = -\frac{Rx^{2\gamma} \rho_0^{\gamma} \zeta}{\eta^{2\gamma} \eta_x^{\gamma} } (\frac{\gamma(\rho_0)_x}{\rho_0}+\frac{\zeta_x}{\zeta}-\gamma (\frac{x^2}{\eta^2\eta_x})_x) =  d^{\gamma \delta}\mathcal{O}(\frac{1}{d}) > 0  \quad \text{in}~(1-d_0,1)\times [0,T], $$
	for small enough $d_0$ depending on $ C_1 $ and $C_2 $ in \eqref{H1-1}.
	Thus, \subeqref{NS-3-Lag}{1} implies that $v$ satisfies the differential inequality :
	\begin{equation} \label{sign-3}
		\begin{aligned}
			(\frac{x^2 \rho_0}{\eta^2} \partial_t +\sL_3)v:= &\Big( \frac{x^2 \rho_0}{\eta^2} \partial_t - \frac{2{\mu}+\lambda }{\eta_x} \partial_{x}^2 - (\frac{2(2\mu+\lambda)}{\eta} -\frac{(2\mu+\lambda)\eta_{xx}}{\eta_x^2} \\
			&   + \frac{(2\mu+\lambda)_x}{\eta_x}) \partial_{x} - (-\frac{2(2\mu+\lambda)\eta_x}{\eta^2} +\frac{2\lambda_x}{\eta})
			\Big)v \\
			=& -(\frac{Rx^{2\gamma} \rho_0^{\gamma} \zeta}{\eta^{2\gamma} \eta_x^{\gamma} })_x> 0 \quad \text{in}~(r_0,1)\times (0,T].
		\end{aligned}
	\end{equation} 
    $\rho_0 \partial_t +\sL_3$ satisfies \eqref{L-coeff} for $0\leq \alpha(\gamma-1)\leq1$  in $ [1-d_0,1]\times[0,T] $ , provided that $ d_0 $ is small enough, by taking
    \begin{equation}\label{L3-coeff}
    	\begin{aligned}
    		& a_0=c_v\frac{x^2 \rho_0}{\eta^2} ,~ a=\frac{2{\mu}+\lambda }{\eta_x},
    		~ c= (-\frac{2(2\mu+\lambda)\eta_x}{\eta^2} +\frac{2\lambda_x}{\eta}),\\
    		&\frac{a_0}{a}=\mathcal{O}(\rho_0^{1-\alpha(\gamma-1)}) <C, \\
    		& \frac{b}{a}=\frac{\eta_x}{2\mu+\lambda } (\frac{2(2\mu+\lambda)}{\eta} -\frac{(2\mu+\lambda)\eta_{xx}}{\eta_x^2} + \frac{(2\mu+\lambda)_x}{\eta_x}),\\
    		&\quad < \frac{(\rho_0^{\alpha (\gamma-1) } )_x}{\rho_0^{\alpha (\gamma-1) } } +C <-\mathcal{O}(\frac{1}{d})+C \quad \text{in} ~ (1-d_0,1)\times [0,T],\\
    		& \frac{c}{a} \sim \mathcal{O}(\frac{(\rho_0^{\alpha (\gamma-1) } )_x}{\rho_0^{\alpha (\gamma-1) } }) \text{~and~thus~} -\frac{Ca}{d}<c\leq 0, ~ \text{in} ~ (1-d_0,1)\times [0,T],
    	\end{aligned}
    \end{equation}
    thanks to \eqref{H1-1} and \eqref{S-assump-3-Lag}. In the case that $b$ is not continuous on $[1-d_0,1]\times [0,T]$, we can make the adjustment as \eqref{L1-coff-re}.

	Thanks to $v\in C^{2,1}([0,1) \times [0,T]) \cap C( [0,1] \times[0,T])$, \eqref{H2*}, \eqref{sign-3}, and $v(1,t)=0$ from Proposition \ref{prop-v-3}, it follows from the strong maximum principle Proposition \ref{prop-strong} (\ref{w-min-prcp-1} in Lemma \ref{lem-weak} is enough) that  $v>0$ in $(r_0,1)\times[0,T]$.
	However, it follows from Hopf's lemma Proposition \ref{prop-Hopf} that $v_x(1,t)<0$ for $t\in[0,T]$, which contradicts Proposition \ref{prop-v-3}. 
	Consequently, Theorem \ref{thm-3-Lag} is proved.
\end{proof}

\section{Proof of Theorem \ref{thm-3'}}\label{sect-proof-thm-3'}

\subsection{Lagrangian formulation} In the case of $ n = 3 $ with spherical symmetry and $k\neq0$, following transform of the flow map of Section \ref{sect-proof-thm-3}, the reduced vacuum free boundary problem \eqref{NS-3} can be written in the Lagrangian coordinates as 
\begin{equation}\label{NS-3-Lag'}
	\begin{cases}
		\frac{x^2\rho_0}{\eta^2} v_t+ (\frac{Rx^{2} \rho_0 \Theta}{\eta^{2} \eta_x })_x =(2\mu+\lambda)(\frac{v_x}{\eta_x}+2\frac{v}{\eta})_x + (2\mu_x\frac{ v_x}{\eta_x} + \lambda_x (\frac{v_x}{\eta_x}+2\frac{v}{\eta})  )   &\text{in}\ [0,1)\times(0,T],\\
		c_v \frac{x^{2} \rho_0}{\eta^{2}  } \Theta_t + \frac{Rx^{2} \rho_0 \Theta}{\eta^{2} \eta_x } \frac{(\eta^2v)_x}{\eta^2 } = \frac{4}{3} \mu \eta_x (\frac{v_x}{\eta_x} -\frac{v}{\eta})^2 &\\
		\qquad\qquad\qquad+ (\frac{2}{3}\mu +\lambda) \eta_x (\frac{v_x}{\eta_x} +2\frac{v}{\eta} )^2 +\kappa ( (\frac{\Theta_x}{\eta_x})_x +2\frac{\Theta_x}{\eta}   )  &\text{in}\ [0,1)\times(0,T],\\
		\Theta>0 &\text{in}\ [0,1)\times(0,T],\\
		\Theta=0,~2 \mu \frac{v_x}{\eta_x} + \lambda (\frac{v_x}{\eta_x} +2\frac{v}{\eta})=0 & \text{at}\ x=1,\\
		(v,\Theta)=(u_0,\theta_0)  & \text{on}\ [0,1)\times\{t=0 \}.
	\end{cases}
\end{equation}
Note that $\theta_0>0$ in $\Omega_0$.
To prove Theorem \ref{thm-3'}, it suffices to prove the following theorem in Lagrangian coordinates:
\begin{theorem}\label{thm-3-Lag'}
	When $\bar{\kappa}>0$ and assume $0\leq \alpha\leq\frac{1}{\gamma-1},$ the solution $(v,\Theta)$ with 
	\begin{equation}
		\label{S-assump-3-Lag'}
		\left\{
		\begin{aligned} 
			&v\in C^2_1([0,1) \times [0,T]) \cap L^{\infty}(0,T; W^{2,\infty} (0,1) ) ,\\
			&v_t \in L^{\infty} (0,T; L^\infty(0,1)),\\
			&\Theta\in C^2_1([0,1) \times [0,T])\cap C([0,1]\times [0,T]),
		\end{aligned}
		\right.
	\end{equation}
	to the free boundary problem in Lagrangian coordinates for any positive time $T$, has to satisfy that 
	\begin{align} 
		& \Theta_x(1,t)<0 \quad t\in[0,T]. \label{physical-vacuum-theta-Lag}
	\end{align}
	Moreover, if initially the first line of \eqref{H1-1} is satisfied and $\theta\in L^{\infty}(0,T; H^1_0 (0,1) )$, then the entropy is bounded in space-time if and only if $\delta=1/(\gamma-1)$.
\end{theorem}

\subsection{Proof of Theorem \ref{thm-3-Lag'}}\label{sect-con-3'}
\begin{proof}[Proof of Theorem \ref{thm-3-Lag'}]
	Suppose that there exists a classical solution $(v,\theta)$ to \eqref{NS-3-Lag} with \eqref{S-assump-3-Lag'}, \eqref{H1-1}, and \eqref{H2*} for some time $T$.
	
	 \subeqref{NS-3-Lag'}{2} implies that $\Theta$ satisfies the differential inequality
	\begin{equation} \label{sign-3'}
		\begin{aligned}
		 &	(c_v\frac{x^2 \rho_0}{\eta^2} \partial_t +\tilde{\sL}_3)\Theta:=\Big( c_v \frac{x^2 \rho_0}{\eta^2} \partial_t - \frac{\kappa }{\eta_x} \partial_{x}^2 - \kappa(\frac{2}{\eta} -\frac{\eta_{xx}}{\eta_x^2}) \partial_{x} 
			   + \frac{Rx^{2} \rho_0}{\eta^{2} \eta_x } \frac{(\eta^2v)_x}{\eta^2 }
			\Big)\Theta\\
		 &\qquad\qquad  = \frac{4}{3} \mu \eta_x (\frac{v_x}{\eta_x} -\frac{v}{\eta})^2 + (\frac{2}{3}\mu +\lambda) \eta_x (\frac{v_x}{\eta_x} +2\frac{v}{\eta} )^2 \geq 0 \quad \text{in}~(0,1)\times [0,T].
		\end{aligned}
	\end{equation} 
    $\rho_0 \partial_t +\tilde{\sL}_3$ satisfies \eqref{L-coeff}, except that the coefficient $c$ satisfies the condition in \cref{cor-strong},  for $0\leq \alpha(\gamma-1)\leq1$ in $ [0,1]\times[0,T] $ by taking
    \begin{equation}\label{L3'-coeff}
    	\begin{aligned}
    		& a_0=\frac{x^2 \rho_0}{\eta^2} ,~ a=\frac{\kappa }{\eta_x},~b=\kappa(\frac{2}{\eta} -\frac{\eta_{xx}}{\eta_x^2}) ,
    		~ c= -\frac{Rx^{2} \rho_0}{\eta^{2} \eta_x } \frac{(\eta^2v)_x}{\eta^2 },\\
    		&\frac{a_0}{a}=\mathcal{O}(\rho_0^{1-\alpha(\gamma-1)}) <C, ~ \frac{b}{a}=\eta_x (\frac{2}{\eta} -\frac{\eta_{xx}}{\eta_x^2})<C, ~ \text{in} ~ [0,1)\times [0,T],\\
    		& \frac{c}{a_0} \sim \mathcal{O}(\frac{(\eta^2v)_x}{\eta^2 }) \text{~and~thus~} |\frac{c}{a_0}|<C, ~ \text{in} ~ (0,1)\times[0,T],
    	\end{aligned}
    \end{equation}
    thanks to \eqref{H1-1} and \eqref{S-assump-3-Lag'}.
    
	Thanks to \eqref{sign-3'}, positivity of $\theta_0$ in $[0,1)\times [0,T]$ and $\Theta(1,t)=0$ from \subeqref{NS-3-Lag'}{4}, it follows from the strong maximum principle, Proposition \ref{prop-strong}, that  $\Theta>0$ in $(0,1)\times[0,T]$, and Hopf's lemma, Proposition \ref{prop-Hopf}, that $\Theta_x(1,t)<0$ for $t\in[0,T]$, which is exactly \eqref{physical-vacuum-theta}.
	
	If initially the first line of \eqref{H1-1} is satisfied and $\theta\in L^{\infty}(0,T; H^1_0 (\Omega(t)) )$,
	then \begin{align*}
		\frac{\bar A}{R}e^{S(\gamma-1)/R}&=\frac{\Theta}{\rho^{\gamma-1}}
		\sim\frac{\Theta}{d^{\delta(\gamma-1)}} \sim \abs{\nabla_{\mathbf{n}} \Theta}{} d^{1-\delta(\gamma-1)} ~\qquad\text{as}~d\rightarrow 0^+,
	\end{align*}
    where $\rho^{\gamma-1}= \rho_0^{\gamma-1} (x^2\eta^{-2}\eta_x^{-1})^{\gamma-1}\sim\rho_0^{\gamma-1} \sim d^{\delta(\gamma-1)}$ since $v\in C^2_1 ([0,1)\times[0,T])$.
    The entropy $S$ is bounded if and only if $0<\abs{\nabla_{\mathbf{n}} \Theta}{} d^{1-\delta(\gamma-1)}<\infty$ as $d\rightarrow 0^+$, i.e., $\delta=1/(\gamma-1)$.
	Consequently, Theorem \ref{thm-3-Lag'} is proved.
\end{proof}


\vline

\section*{Acknowledgements}
Xin Liu would like to thank the Isaac Newton Institute for Mathematical Sciences for support and hospitality during the programme TUR when part of this work was under taken. 
The research of Xin Liu was partially supported by a grant from the Simons Foundation during his visit to the Isaac Newton Institute, and by the Deutsche Forschungsgemeinschaft (DFG). 

The research of Yuan Yuan is supported by the National Nature Science Foundation of China (Nos. 11901208 and 11971009), the Natural Science Foundation of Guangdong Province (No. 2021A1515010247), and National Key R\&D Program of China (No. 2021YFA1002900).

\vline


\bibliographystyle{amsplain}
\providecommand{\bysame}{\leavevmode\hbox to3em{\hrulefill}\thinspace}
\providecommand{\MR}{\relax\ifhmode\unskip\space\fi MR }
\providecommand{\MRhref}[2]{%
  \href{http://www.ams.org/mathscinet-getitem?mr=#1}{#2}
}
\providecommand{\href}[2]{#2}

\end{document}